\documentclass[leqno,12pt]{article}
\usepackage{amsmath, amssymb}
\usepackage[dvips]{graphicx}
\usepackage{amsthm}

\usepackage{color}
\usepackage{ulem}
\theoremstyle{plain} 
\newtheorem{theorem}{\indent\sc Theorem}[section] 
\newtheorem{lemma}[theorem]{\indent\sc Lemma}

\theoremstyle{definition} 

{\tiny {\scriptsize }}

\title{Elementray proof of a Lemma for parallel mean curvature surfaces in complex space forms} 
\author{ Katsuei  Kenmotsu }

\makeatletter

\@addtoreset{equation}{section}
\makeatother

\date{}

\begin{document}
\maketitle

\begin{center} 
	To Marcos Dajczer on the occaison of his 75th birthday 
\end{center}

\begin{abstract}
	We provide an elementary proof of a lemma that plays an important role in the classification of parallel mean curvature surfaces in two-dimensional  complex space forms.
\end{abstract}

\section{Introduction}
In the theory of parallel mean curvature surfaces in  complex  space forms, Lemma 1 of Kenmotsu \cite{ken1} plays an important rolle for the classfication of these surfaces when the codimension is two (see   Kenmotsu \cite{ken2}).  The Lemma claims that a part of the compxified 2nd fundamental tensor of the surface is a function of the Kaehler angle of the surface, if the ambiant space is non-Euclidean and  the Kaehler angle is not constant.   The proof of the Lemma in \cite{ken1} uses the exterior derivatives of forms and a skillful substitution of the 2nd fundamental tensor. In this paper, we give an elementary proof of the Lemma 1 of \cite{ken1} by  using partial differentiations only.

 \section{The Lemma}
The notations and structure equations in this paper are consistent with those in Kenmotsu \cite{ken1}. In this context, $M^2$ represents a two-dimensional Riemannian manifold, and $x$ denotes an isometric immersion of $M^2$ into a complex space form $\overline{M}[4\rho]$ of complex dimension two, where $4\rho$ represents the holomorphic sectional curvature of the complex space form. In this paper, we study the immersion $x$ such that the mean curvature vector of the immersion, denoted as $H$, is non-zero and parallel when considering the normal connection on the normal bundle of the immersion. 

Since $x$ is not minimal,  $x$ is neither holomorphic nor anti-holomorphic on $M^2$. Hence, the Kaehler angle of $x$, denoted as $\alpha$, satisfies $\sin \alpha \neq 0$ in a neighborhood of any point  on $M^2$.
Since our study here only pertains to local properties of the immersion, we may consider this neighborhood to be the entirety of $M^2$.

Given that we utilize formulas (2.1) through (2.6) (excluding (2.3)) from Kenmotsu  \cite{ken1}, for the sake of completeness, we provide them below. There exists a local field of complex one-form, denoted as $\phi$, on $M^2$ such that the Riemannian metric $ds^2$ on $M^2$ can be expressed as $ds^2 = \phi\overline{\phi}$ (cf. Chern and Wolfson \cite{cherwolf}). Additionally, let $a$ and $c$ be complex-valued functions on $M^2$ that determine the second fundamental tensors of the immersion.
Then,  the Kaehler angle $\alpha$ and the complex 1-form $\phi$  satisfy
\begin{eqnarray}
d\alpha &=& ( a +  b )\phi + (\bar{a} + b )\bar{\phi} , \\
d\phi &=&  (\bar{a} - b)\cot \alpha \cdot \phi \wedge \bar{\phi},
\end{eqnarray}
where $2b = |H| >0$. The complexified 2nd fundamental tensor $a$ and $c$  satisfy
 \begin{eqnarray}
&& da\wedge\phi= -\left(2a (\bar{a} - b) \cot\alpha  + \frac{3}{2}\rho\sin \alpha \cos \alpha \right)
\phi\wedge\bar{\phi}, \\
&& dc\wedge\bar{\phi}=2c(a - b)\cot\alpha \cdot \phi\wedge\bar{\phi}, \\
&& |c|^{2} = |a|^{2} + \frac{\rho}{2}(-2 + 3\sin^{2}\alpha),
\end{eqnarray}
where (2.3) and (2.4) are the Codazzi-Mainardi equations, and (2.5) is the Ricci equation of $x$. 
 The Gauss equation of the immersion is not used in this paper.

Now we prove that the Kaehler angle $\alpha$ can be taken as a coordinate function on $M^2$, if $\alpha$ is not constant.
\begin{lemma}
 Let $x$ be an isometric immersion of $M^2$ into $\overline{M}[4\rho]$ 
with non-zero parallel mean curvature vector.  If the Kaehler angle $\alpha$ is not constant, then  there exists a real valued function  $\beta$ on $M^2$ such that $(\alpha,\beta)$ is a  local coordinate system on $M^{2}$.
\end{lemma}
Proof.
Let  $(u,v)$   be the isothermal local coordinates on $M^2$ such that $\phi= \lambda(du+idv),\ (i^2=-1)$. 
Suppose that the Kaehler angle $\alpha$ on $M^2$ is not constant on $M^2$, so we may assume $d\alpha \neq 0$ on $M^2$, because we study here only local properties of the immersion. By (2.1), we note that $a+b \neq 0$ on $M^2$.  Let a real valued function $k(u,v)\ (\neq 0)$ defined on a neighborhood of any point of $M^2$ be   a solution of a linear  first order partial differential equation 
$
\alpha_{u}k_{u} +\alpha_{v} k_{v}+(\alpha_{uu} + \alpha_{vv}) k =0.
$ (There exists always such a solution of the equation  by the famous existence theorem.) \  We consider a system of partial differential equations defined by  $\beta_{u}=k\alpha_{v}, \ \beta_{v}=-k\alpha_{u}$.  Since this is completely integrable by the property of $k(u,v)$,  there exists a real valued solution $\beta = \beta(u,v)$ of the system in a neighborhood of any point of $M^2$.
Since  $\alpha_{u}=\lambda(a+b) +\overline{\lambda}(\overline{a}+b) , \alpha_{v}=i(\lambda(a+b) -\overline{\lambda}(\overline{a}+b))$ by (2.1), it holds that
$
d\beta = ik((a+b)\phi - (\overline{a}+b)\overline{\phi}),
$
which implies
\begin{equation}
	\phi = \frac{1}{2k(a+b)}(kd\alpha - i d\beta),\quad \phi \wedge \bar{\phi} = \frac{i}{2k|a+b|^2}d\alpha \wedge d\beta,
\end{equation}
on a neighborhood of any point in $M^2$. Hence $(\alpha,\beta)$ is a local  coordinate system on the neighborhood, proving Lemma 2.1.

Hereafter, the quantities $k,a,$ and $c$ are considered as functions of $\alpha$ and $\beta$.

The local structure of the  immersion satisfying $a =\bar{a}$ on $M^2$ has been determined  in Kenmotsu and Zhou \cite{kenzhou} and Kenmotsu \cite{ken1}.   Hirakawa \cite{hirakawa} gave  another proof for this case. Also, the immersions with $d\alpha =0$ on $M^2$  are classified by Hirakawa \cite{hirakawa}.
From this point on, we will  study  an immersion satisfying $d\alpha \neq 0$ and  $a \neq \bar{a}$  at a point of $M^2$.
An immersion $x$ is called  of general type if it satisfies these two conditions  on $M^2$. The Lemma 1 in \cite{ken1} is stated as
\begin{lemma}  Let $x$ be an isometric immersion of $M^2$ into $\overline{M}[4\rho]$ 
	with non-zero parallel mean curvature vector. Suppose that  $\rho \neq 0$ and $x$ is of   general type.  Then, $d\alpha \wedge da =0$ on $M^2$.
\end{lemma}
 Lemma 2.2 implies that if $x$ is of  general type, then $a$ is a function of $\alpha$ only and this
 played a crucial  rolle to determine the local structure of these surfaces (see Kenmotsu \cite{ken2}). 

\section{Elementary proof}
The original proof of Lemma 2.2 (referred to as Lemma 1 in \cite{ken1}) utilizes the theory of exterior differential forms and employs a skillful substitution for the partial derivatives of  $a$. In this paper, we  present an elementary proof of Lemma 2.2, relying solely on the use of the Kaehler angle as a coordinate function of the suface and on partial differentiations, except for the proof of equation (3.1). Our main idea involves utilizing the complex conjugate to calculate the difference in the higher-order  derivatives  for  $\alpha$ and  $\beta$ of functions  $a$ and  $c$ ( see, for example,  the proof of Lemma 3.1). It shall be remarked that all rigorous computations of the derivatives for functions $p_{i}$  ( $i=1,2,... $), as defined in this paper, were carried out using the symbolic manipulation program Mathematica by Wolfram. We do not provide the Mathematica program codes as they are all simple,  consisting only of rational functions and their differentiations.

Let us clear  the first  equation  in (2.6) of fractions and take the exterior derivative.
 By (2.1), (2.2) and (2.3), we have
\begin{equation}
k_{\alpha} = - p_{1}k, 
\end{equation}
where $p_{i}, (i=1,2,...)$ are functions of  $\alpha,a$, and $\bar{a}$ listed in the Appendix.  Equations  (2.3) and (2.4) are rewritten  as functions of $\alpha$ and $\beta$ 
\begin{eqnarray}
a_{\alpha} - ika_{\beta} &=& p_{2}, \\
c_{\alpha} + ikc_{\beta} &=& 2c p_{3}.
\end{eqnarray}

\begin{lemma}   Let $x$ be an isometric immersion of $M^2$ into $\overline{M}[4\rho]$ 
	with non-zero parallel mean curvature vector.  If the Kaehler angle $\alpha$ is not constant, then  it holds that
	\begin{equation}
		\bar{c}c_{\alpha} - a\bar{a}_\alpha = p_{4}.
	\end{equation}
\end{lemma}
Proof. Differentiating  partially   (2.5) for $\beta$ and multiplying the formula by $ik$, we see that
$
\bar{c}(ikc_{\beta}) - c \overline{(ikc_{\beta}}) = \bar{a}(ika_{\beta}) -a \overline{(ika_{\beta})}.
$
By (3.2) and (3.3), this  is written as 
$$
-\bar{c}c_{\alpha} + c \bar{c}_{\alpha} = \bar{a}a_{\alpha} - a\bar{a}_{\alpha} -2|c|^2(p_{3}-\bar{p}_{3}) - \bar{a}p_{2} + a\bar{p}_{2}.
$$
Differentiating  partially  (2.5) for $\alpha$ yields
$
\bar{c}c_{\alpha} + c \bar{c}_{\alpha} = \bar{a}a_{\alpha} + a\bar{a}_{\alpha} + 3 \rho \sin \alpha \cos \alpha.
$
Lemma 3.1 follows from these two formulas.

It follows from  (3.2), (3.3) and (3.4) that
\begin{equation}
		\bar{c}c_{\beta} - a\bar{a}_\beta = \frac{i}{k}p_{5}.
\end{equation}
Differentiating partially (3.4) for $\alpha$ proves that
\begin{equation}
	|c_{\alpha}|^2 - |a_{\alpha}|^2 + \bar{c}c_{\alpha\alpha} - a \overline{a_{\alpha\alpha}} =
	\frac{\partial p_{4}}{\partial \alpha} + 	\frac{\partial p_{4}}{\partial a}a_{\alpha}  +	\frac{\partial p_{4}}{\partial \bar{a}}\overline{a_{\alpha}}.
\end{equation}
We want to find a formula of  $\bar{c}c_{\alpha\alpha} - a \overline{a_{\alpha\alpha}}$ in terms of $a_{\alpha}$. Computing   partial derivatives of (3.3) and (3.2) for $\alpha$, we have
\begin{eqnarray}
		c_{\alpha\alpha} + ikc_{\beta\alpha} &=& p_{1}ikc_{\beta} + 2 p_{3} c_{\alpha}
		+2c\left(\frac{\partial p_{3}}{\partial \alpha} + 	\frac{\partial p_{3}}{\partial a}a_{\alpha}  +	\frac{\partial p_{3}}{\partial \bar{a}}\overline{a_{\alpha}} \right), \\
	\overline{	a_{\alpha\alpha} }+ ik\overline{a_{\beta\alpha} }
	&=& \bar{p}_{1}ik\overline{a_{\beta} }
	+\overline{\frac{\partial p_{2}}{\partial \alpha} }+ 	\overline{\frac{\partial p_{2}}{\partial \bar{a}}}a_{\alpha}  +	\overline{\frac{\partial p_{2}}{\partial a}}\overline{a_{\alpha}}. 
\end{eqnarray}
It is remarked here that
\begin{equation}
p_{1} - \overline{p_{1}}=0, \quad	\frac{\partial p_{3}}{\partial \bar{a}} =0,\quad 
 \overline{	\frac{\partial p_{2}}{\partial a}} = 2p_{3}.
\end{equation}
 It follows from  (3.4), (3.5), (3.7), (3.8) and (3.9) that
\begin{eqnarray*}
\bar{c}c_{\alpha\alpha} - a \overline{a_{\alpha\alpha}}&=& ik(a\overline{a_{\beta\alpha} }- \bar{c}c_{\beta\alpha}) +\left( - a \overline{\frac{\partial p_{2}}{\partial \bar{a}}} +   2|c|^2 \frac{\partial p_{3}}{\partial a}  \right)a_{\alpha}  
   \\
 && -p_{1}p_{5} + 2p_{3}p_{4} 
-a \overline{ \frac{\partial p_{2}}{\partial \alpha} } + 2|c|^2 \frac{\partial p_{3}}{\partial \alpha}. 
\end{eqnarray*}
Let the both sides of   (3.4)  differentiate partially  for $\beta$. Then  multiplying the formula by $ik$,  (3.2) and (3.3) imply 
\begin{eqnarray*}
ik(a\overline{a_{\alpha\beta}} - \bar{c}c_{\alpha\beta})
&= & |c_{\alpha}|^2 - |a_{\alpha}|^2  -  \frac{\partial p_{4}}{\partial a}a_{\alpha} + \left(p_{2}-2a\bar{p}_{3} +\frac{\partial p_{4}}{\partial \bar{a}} \right)\overline{a_{\alpha}} \\
&& \quad - 2 \bar{p}_{3}p_{4}+p_{2} \frac{\partial p_{4}}{\partial a } 
 - \bar{p}_{2}\frac{\partial p_{4}}{\partial \bar{a}}.
\end{eqnarray*}
It follows from above two formulas that
\begin{eqnarray*}
\bar{c}c_{\alpha\alpha} - a \overline{a_{\alpha\alpha}}&=&  |c_{\alpha}|^2 - |a_{\alpha}|^2  +
\left(-a\overline{ \frac{\partial p_{2}}{\partial \bar{a}}} + 2|c|^2  \frac{\partial p_{3}}{\partial a}
 - \frac{\partial p_{4}}{\partial a}\right) a_{\alpha}  +\left (p_{2} - 2a\overline{p_{3}}+ \frac{\partial p_{4}}{\partial \bar{a}} \right)\overline{a_{\alpha}}     \\
&&   +2(p_{3} - \bar{p}_{3})p_{4}  -p_{1}p_{5}  -a\overline{\frac{\partial p_{2}}{\partial \alpha}} + 2|c|^2  \frac{\partial p_{3}}{\partial \alpha} +p_{2}  \frac{\partial p_{4}}{\partial a}- \bar{p}_{2}\frac{\partial p_{4}}{\partial \bar{a}} .
\end{eqnarray*}
Substiuting this for (3.6) shows that
\begin{eqnarray*}
 |c_{\alpha}|^2 - |a_{\alpha}|^2 = \left(\frac{1}{2}a\overline{ \frac{\partial p_{2}}{\partial \bar{a}}}  -|c|^2  \frac{\partial p_{3}}{\partial a} +\frac{\partial p_{4}}{\partial a} \right) a_{\alpha}  + \left(-\frac{1}{2}p_{2} + a\bar{p}_{3}  \right )\overline{a_{\alpha}} + p_{6}.
\end{eqnarray*}
The coefficients of $a_{\alpha}$ and $\overline{a_{\alpha}}$ are simplified as 
\begin{eqnarray*}
	\frac{a}{2}\overline{ \frac{\partial p_{2}}{\partial \bar{a}}} - |c|^2 \frac{\partial p_{3}}{\partial a}
	+ \frac{\partial p_{4}}{\partial a} = -\frac{1}{2}\overline{p_{2}} + \bar{a}p_{3} = -\frac{3\rho \sin \alpha \cos \alpha}{4(a+b)}, 
\end{eqnarray*}
so that
\begin{equation*}
|c_{\alpha}|^2 - |a_{\alpha}|^2 = -\frac{3\rho \sin \alpha \cos \alpha}{4(a+b)} a_{\alpha}  -\frac{3\rho \sin \alpha \cos \alpha}{4(\bar{a}+b)} \overline{a_{\alpha} }+ p_{6}.
\end{equation*}
Since $c_{\alpha}$ is expressed by $\overline{a_{\alpha}},a, \bar{a}$ and $\bar{c}$ by Lemma 3.1,  the above formula proves that
\begin{lemma}   Let $x$ be an isometric immersion of $M^2$ into $\overline{M}[4\rho], (\rho \neq 0),$  with non-zero parallel mean curvature vector. If $x$ is of  general type, then we have
	\begin{equation}
	|a_{\alpha}|^2 = p_{7}a_{\alpha} + \overline{p_{7}a_{\alpha}} +p_{8}.
	\end{equation}
\end{lemma}
Note that $p_{7}$ and $p_{8}$ are well-defined under the assumptions of $\rho \neq 0$ and $d\alpha \neq 0$, and	for  later use, 
	\begin{equation}
	 \overline{ \frac{\partial p_{2}}{\partial \bar{a}}} =	2\frac{\partial p_{7}}{\partial a}, \quad p_{8} = \bar{p}_{8}.
	\end{equation}

	The next task is to find representation formulas of $a_{\alpha\alpha}$ and $a_{\alpha\beta}$  in terms of  $a_{\alpha}$ and   $\bar{a}_{\alpha}$. Differentiating  partially  (3.10) for $\alpha$ proves that
	\begin{equation}
a_{\alpha\alpha} (\overline{a_{\alpha}} - p_{7})  + (a_{\alpha} - \bar{p}_{7}) \overline{a_{\alpha\alpha} } = \frac{\partial p_{7}}{\partial a}a_{\alpha}^{2} +\overline{  \frac{\partial p_{7}}{\partial a}a_{\alpha}^{2}} + p_{9a}a_{\alpha}  + \overline{p_{9a} a_{\alpha}}  + p_{10a}.
		\end{equation}
		Differentiating  partially  (3.10) for $\beta$ and multiplying the formula by $ik$ prove that
		\begin{eqnarray}
&&\quad 	ika_{\alpha\beta}(\overline{a_{\alpha}} - p_{7}) -   (a_{\alpha} - \overline{p_{7}})  \overline{(ika_{\alpha\beta})}   \\
&& =\left( \frac{\partial p_{7}}{\partial a}ika_{\beta} -   \frac{\partial p_{7}}{\partial \bar{a}} \overline{ ( ika_{\beta} )  } \right)a_{\alpha}       -  \left(\overline{ \frac{\partial p_{7}}{\partial a}ika_{\beta}} -  \overline{ \frac{\partial p_{7}}{\partial \bar{a}} } ika_{\beta} \right) \overline{a_{\alpha} } +    \frac{\partial p_{8}}{\partial a}ika_{\beta} -   \frac{\partial p_{8}}{\partial \bar{a}}\overline{(  ika_{\beta} )}.    \nonumber
		\end{eqnarray}
		Differentiating  partially (3.2) for $\alpha$ and using (3.1) and (3.2) prove that
		\begin{equation}
		ika_{\beta\alpha}	=a_{\alpha\alpha} +\left(p_{1}- \frac{\partial p_{2}}{\partial a} \right)a_{\alpha} - 
		\frac{\partial p_{2}}{\partial \bar{a}}\overline{a_{\alpha}} - p_{1}p_{2} - \frac{\partial p_{2}}{\partial \alpha}.
		\end{equation}
		Substituting (3.2) and (3.14) for  (3.13) and then using (3.10) show that
		\begin{eqnarray}
	&&		a_{\alpha\alpha} (\overline{a_{\alpha}} - p_{7})  - \overline{	a_{\alpha\alpha}} (a_{\alpha} - \overline{p}_{7})   \nonumber \\
&&	 = -\left ( \overline{\frac{\partial p_{2}}{\partial \bar{a}}} - \frac{\partial p_{7}}{\partial a}\right)a_{\alpha}^{2} +\left( \frac{\partial p_{2}}{\partial \bar{a}} -\overline{  \frac{\partial p_{7}}{\partial a}}\right) \overline{a_{\alpha}}^{2} + p_{9b}a_{\alpha}  - \overline{p_{9b} a_{\alpha}}  + p_{10b}.
\nonumber
	\end{eqnarray}
	Adding this to (3.12) and using (3.11) show that
\begin{equation}
	2a_{\alpha\alpha} (\overline{a_{\alpha}} - p_{7})	=\frac{\partial p_{2}}{\partial \bar{a}} \overline{a_{\alpha}}^{2} +(p_{9a} + p_{9b}) a_{\alpha} + (\overline{p_{9a} }- \overline{p_{9b}})\overline{a_{\alpha}} +p_{10a} + p_{10b} .
\end{equation}
By  the identity $(\overline{a_{\alpha}} - p_{7})(a_{\alpha} -\overline{ p_{7}}) = |p_{7}|^2 + p_{8}$,  when $\overline{a_{\alpha}} - p_{7} \neq 0$, the above formula is written as	
	$$
2a_{\alpha\alpha} = \frac{ (a_{\alpha} - \overline{p_{7}})}{ |p_{7}|^2+p_{8}}\left(\frac{\partial p_{2}}{\partial \bar{a}} \overline{a_{\alpha}}^{2} +(p_{9a} + p_{9b}) a_{\alpha} + (\overline{p_{9a} }- \overline{p_{9b}}) \overline{a_{\alpha}} +p_{10a} + p_{10b}  \right).
$$
	Expanding  the right-hand side formula above,  the identity
	\begin{equation}
	a_{\alpha}\overline{a_{\alpha} }^2 = |a_{\alpha}|^2 \overline{a_{\alpha}}
	 = p_{7}^2 a_{\alpha}  + \overline{p}_{7} \overline{a_{\alpha}}^2 + (|p_{7}|^2 + p_{8}) \overline{a_{\alpha} }+ p_{7}p_{8}
			\end{equation}
			proves the following Lemma 3.3 
	\begin{lemma}  Let $x$ be an isometric immersion of $M^2$ into $\overline{M}[4\rho], (\rho \neq 0), $ 	with non-zero parallel mean curvature vector. Suppose that  $x$ is of  general type and $\overline{a_{\alpha} }- p_{7}\neq 0$.  Then 
	\begin{equation}
  a_{\alpha\alpha} = p_{11}a_{\alpha}^2 + p_{12}a_{\alpha} +\frac{1}{2} \frac{\partial p_{2}}{\partial \bar{a}} \overline{a_{\alpha}} + p_{13}.
	\end{equation}
		\end{lemma}

	
		When  $\overline{a_{\alpha} }- p_{7} =0$ on an open set of $M^2$, it follows from (3.2) that 	\begin{eqnarray}
			ik\frac{\partial a_{\alpha}}{\partial \beta} &=& ik\left(\frac{\partial\overline{ p_{7}}}{\partial a} a_{\beta} + \frac{\partial\overline{ p_{7}}}{\partial \overline{a}} \overline{a_{\beta}}\right) =  \frac{\partial\overline{ p_{7}}}{\partial a} (ika_{\beta}) -  \frac{\partial\overline{ p_{7}}}{\partial \overline{a}} (\overline{ika_{\beta}}) \nonumber  \\
			&=&  \frac{\partial\overline{ p_{7}}}{\partial a} (\overline{ p_{7}} - p_{2}) -  \frac{\partial\overline{ p_{7}}}{\partial \overline{a}} ( p_{7} - \overline{p_{2}}).  
		\end{eqnarray}	
		On the other hand, by (3.14), we have
		\begin{eqnarray}
			ik\frac{\partial a_{\beta}}{\partial \alpha} &=& \frac{\partial \overline{p_{7}}}{\partial \alpha} + \frac{\partial \overline{p_{7}}}{\partial a}\overline{ p_{7}}+ \frac{\partial \overline{p_{7}}}{\partial \overline{a}}  p_{7} +\left(p_{1}- \frac{\partial p_{2}}{\partial a} \right)\overline{ p_{7}}    \nonumber \\
			&& - \frac{\partial p_{2}}{\partial \overline{a}} p_{7}- p_{1}p_{2} - \frac{\partial p_{2}}{\partial \alpha}.  
		\end{eqnarray}
		Since $a_{\alpha \beta} = a_{\beta\alpha}$,  it follows from  (3.18),  (3.19), (3.9) and (3.11)     that
		\begin{equation}
			p_{1}p_{2}  - p_{1}\overline{p_{7}} + 2\overline{p_{3}p_{7}} + \frac{\partial p_{2}}{\partial \alpha}   + \frac{1}{2} \overline{p_{2}}\frac{\partial p_{2}}{\partial \overline{a}}   - \frac{\partial \overline{p_{7}}}{\partial \alpha} -p_{2} \frac{\partial \overline{p_{7}}}{\partial a}=0,
		\end{equation}
		which is an equation for $\alpha,a$ and $\overline{a}$.
		\begin{lemma}
			Let $x$ be an isometric immersion of $M^2$ into $\overline{M}[4\rho], (\rho \neq 0), $ 	with non-zero parallel mean curvature vector. Suppose that $x$ is of general type and $\overline{a_{\alpha} }- p_{7} =0$ on an open set of $M^2$. Then $a$ is a function of $\alpha$.
		\end{lemma}
		\begin{proof}
			Let $F(\alpha,a, \overline{a})$ be the left-hand side formula of (3.20). We observe that $F(\alpha,a, \overline{a})$ is non-trivial because of $F(\pi/4,0,0) = 15\rho/8b  \neq 0$. If $F_{a}= F_{\overline{a}}=0$, then $F(\alpha,a, \overline{a})$ is a function of $\alpha$ only.  Equation (3.20) implies that $\alpha$ is constant, leading to a  contradiction. So, $F_{a}$ or $F_{\overline{a}}$ can not be   zero. When $F_{\overline{a}}$ is not zero, $\overline{a}$ is a function of $\alpha$ and $a$, denoted as  $\overline{a} = f(\alpha,a)$,  by the implicit function theorem.  The functin $f(\alpha,a)$ is not real-valued;  otherwise, we would have $\overline{a} =a$, which results in a  contradiction. Taking  the complex conjugate of the equation  $\overline{a} = f(\alpha,a)$ yields   $a= \overline{ f(\alpha,a)} = g(\alpha,\overline{a})$ for some function $g$.  Define  $K(\alpha, a) = a- g(\alpha,f(\alpha,a))$. By replacing  $F(\alpha,a, \overline{a})$ with  $K(\alpha, a) $ and using   similar arguments as above,  $K(\alpha, a)=0$ is solved with respect to $a$,  proving Lemma 3.4.
			When  $F_{a}  \neq 0$,  Lemma 3.4 is proved  in the same mannar as before.  
		\end{proof}

		From now on we study the case of  $\overline{a_{\alpha} }- p_{7}\neq 0$. 	It follows from Lemma 3.3 and (3.14)  that
		\begin{equation}
			ika_{\beta\alpha}	= p_{11}a_{\alpha}^2  + p_{14}a_{\alpha} -\frac{1}{2} \frac{\partial p_{2}}{\partial \bar{a}} \overline{a_{\alpha} } + p_{15}.
		\end{equation}
In order to	 get the difference of  the 2nd derivatives of $a_{\alpha}$, we calculate $a_{\alpha\alpha\beta}$ by (3.17) as follows:
		\begin{eqnarray*}
	&&	ik\frac{\partial}{\partial \beta}a_{\alpha\alpha}=(2p_{11}a_{\alpha} + p_{12})ika_{\alpha\beta} -\frac{1}{2} \frac{\partial p_{2}}{\partial \bar{a}}\overline{( ika_{\alpha\beta)}}  \\
	&&\qquad \qquad			+	\left (\frac{\partial p_{11}}{\partial a}ika_{\beta}   
					- \frac{\partial p_{11}}{\partial \bar{a}}\overline{(ika_{\beta})} \right) a_{\alpha}^2 + \frac{\partial p_{13}}{\partial a} ika_{\beta}  - \frac{\partial p_{13}}{\partial \bar{a}} \overline{(ika_{\beta})}    \\
	&& \qquad \qquad   + \left(\frac{\partial p_{12}}{\partial a} ika_{\beta}  - \frac{\partial p_{12}}{\partial \bar{a}} \overline{(ika_{\beta} )  } \right)a_{\alpha} + \frac{1}{2}\left(\frac{\partial^2 p_{2}}{\partial \bar{a}\partial\partial a} ika_{\beta}  - \frac{\partial^2 p_{2}}{\partial \bar{a}^2} \overline{(ika_{\beta} )  } \right)\overline{a_{\alpha} } .
			\end{eqnarray*} 
				
		It follows from  (3.2), (3.21)  and  the complex conjugate of (3.16) that 
		\begin{eqnarray*}	
	&&	ik\frac{\partial}{\partial \beta}a_{\alpha\alpha} 	= \left(2p_{11}^2 +\frac{\partial p_{11}}{\partial a} \right)a_{\alpha}^3 + p_{16a}a_{\alpha}^2 -\frac{1}{2} \left(\bar{p}_{11}\frac{\partial p_{2}}{\partial \bar{a}} +\frac{\partial^2 p_{2}}{\partial \bar{a}^2} \right)\overline{a_{\alpha}}^2  \\
&&\qquad \qquad \qquad		+ p_{17a} a_{\alpha} + p_{18a}\bar{a}_{\alpha} +p_{19a}. \nonumber
		\end{eqnarray*}
	By the 	similar way, we show  by (3.1),  (3.17) and (3.21) that
		\begin{eqnarray*}
			ik\frac{\partial}{\partial \alpha}a_{\beta\alpha} &=&   \frac{\partial}{\partial \alpha} 	(ika_{\beta\alpha} )  - ik_{\alpha}a_ {\beta\alpha}  \nonumber \\
		&=&\left(2p_{11}^2 +\frac{\partial p_{11}}{\partial a} \right)a_{\alpha}^3 + p_{16b}a_{\alpha}^2 -\frac{1}{2} \left(\bar{p}_{11}\frac{\partial p_{2}}{\partial \bar{a}} +\frac{\partial^2 p_{2}}{\partial \bar{a}^2} \right)\overline{a_{\alpha}}^2   \\
	&& 	 + p_{17b} a_{\alpha} + p_{18b}\bar{a}_{\alpha} +p_{19b}. \nonumber
		\end{eqnarray*}
		Since $(a_ {\alpha\alpha} )_{\beta} =(a_ {\beta\alpha})_{\alpha}$, the above two formulas imply
		\begin{equation}
			p_{16}a_{\alpha}^2 + p_{17}a_{\alpha} + p_{18}\overline{a_{\alpha}} + p_{19} =0.
				\end{equation}
	 This is a non-trivial equation for $a_{\alpha}$, because our system consiting of  (2.3), (2.4)  and (2.5)  is over-determined. The rigorus proof that  (3.22) is non-trivial is shown  by $p_{16} (\pi/4,0,0) \neq 0$ or  $p_{16} (\pi/3,0,0) \neq 0$. By (3.22), $a_{\alpha}$ is determined by $\alpha, a$ and $\overline{a}$,  denoted as  $a_{\alpha}= p_{23}(\alpha,a,\overline{a})$,  because  all coefficients of (3.22) are functions of  $\alpha, a$ and $\overline{a}$.  Furthermore, $p_{23}$ is unique up to the complex conjugate that is later shown in  the proof of Lemma 4.1

	The following is the main result of this paper.
\begin{lemma}   Let $x$ be an isometric immersion of $M^2$ into $\overline{M}[4\rho], (\rho \neq 0),$ with non-zero parallel mean curvature vector. If $x$ is of  general type, then $a$ is a function of $\alpha$.
\end{lemma}
	\begin{proof} By Lemma 3.4, we may assume  $a_{\alpha} - \overline{p_{7}} \neq 0$.  Replacing $\overline{p_{7}}$ in (3.18) and  (3.19) with $p_{23}$,  we have
	$$
	p_{1}p_{2} + \frac{\partial p_{2}}{\partial \alpha} - \left(p_{1} - \frac{\partial p_{2}}{\partial a} \right)p_{23} +   \frac{\partial p_{2}}{\partial \overline{a}} \overline{ p_{23}} -   \frac{\partial p_{23}}{\partial \alpha}  - p_{2} \frac{\partial p_{23}}{\partial a}  + \left(\overline{p_{2}} - 2\overline{ p_{23}} \right) \frac{\partial p_{23}}{\partial \overline{a} } =0.
	$$	
	Let $G(\alpha,a,\overline{a})$ be the left-hand side formula above.	
	This is non-trivial, because the system of (2.3), (2.4) and (2.5) is over-determined. For the rigorus proof,  see Lemma 4.1 in  Appendix.   We prove Lemma 3.5 in the same manner as the proof of Lemma 3.4  replacing  $F (\alpha,a,\overline{a})$ with $G(\alpha,a,\overline{a})$.
\end{proof}

By Lemma 3,5, if $x$ is of general type, then $a_{\beta}=0$ and $a'(\alpha)= p_{2}(\alpha,a,\bar{a})$ by  (3.2). Based on this, the classification of such immersions is carried  out in \cite{ken2}.

\section{Appendix}

	\begin{eqnarray*}
		&& p_{1}(\alpha,a,\bar{a}) = \frac{1}{|a+b|^2}\left(|a-b|^2 + \frac{3}{2}\rho \sin^2\alpha\right)\cot \alpha, \\
		&& p_{2}(\alpha,a,\bar{a}) = \frac{1}{\bar{a}+b}\left( 2a(\bar{a}-b)
		+ \frac{3}{2}\rho \sin^2 \alpha\right) \cot\alpha  ,  \\
	&&	 p_{3}(\alpha,a,\bar{a}) = \frac{a-b}{a+b} \cot\alpha  ,  \\
	&& p_{4}(\alpha,a,\bar{a}) = \left(a\bar{a} + \frac{\rho}{2}(-2+3\sin^2\alpha)\right)(p_{3} - \bar{p}_{3}) + \frac{1}{2}(\bar{a}p_{2} -a \bar{p}_{2}) + \frac{3}{2}\rho \sin\alpha\cos\alpha, \\
	&& p_{5}(\alpha,a,\bar{a}) =  a \bar{p}_{2}  - 2\left(a\bar{a} 	+ \frac{\rho}{2}(-2+3\sin^2\alpha)\right) p_{3}+ p_{4},  \\ 
		&& p_{6}(\alpha, a,\bar{a}) =p_{4}(\bar{p}_{3} - p_{3})  -\left(a\bar{a} 	+ \frac{\rho}{2}(-2+3\sin^2\alpha)\right) \frac{\partial p_{3}}{\partial \alpha}  \\ 
	&& \qquad \qquad \qquad  + \frac{1}{2}\left(p_{1}p_{5}  + a \overline{\frac{\partial p_{2}}{\partial \alpha}} 
	+ \frac{\partial p_{4}}{\partial \alpha} -p_{2} \frac{\partial p_{4}}{\partial a} 
	+ \bar{p}_{2}\frac{\partial p_{4}}{\partial \bar{a}} \right), \\
	&& p_{7}(\alpha,a,\bar{a}) = \frac{1}{\frac{\rho}{2}(-2+3\sin^2\alpha)}\left( \bar{a}p_{4} + \frac{3 \rho \sin \alpha \cos \alpha}{4(a+b)}  \left(a\bar{a} 	+ \frac{\rho}{2}(-2+3\sin^2\alpha)\right) \right),  \\
	&& p_{8}(\alpha,a,\bar{a}) =\frac{1}{\frac{\rho}{2}(-2+3\sin^2\alpha)}\left(  |p_{4}|^2 -\left(a\bar{a} 	+ \frac{\rho}{2}(-2+3\sin^2\alpha)\right)p_{6}  \right),  \\
&& p_{9a}(\alpha,a,\bar{a}) =  \frac{\partial p_{7}}{\partial \alpha} 
+\left (\frac{\partial p_{7}}{\partial \bar{a} } + \overline{\frac{\partial p_{7}}{\partial \bar{a}}} \right)p_{7}
+\frac{\partial p_{8}}{\partial a},\\
&& p_{9b} (\alpha,a,\bar{a})= p_{1}(- \bar{p}_{2} +p_{7} ) -\overline{\frac{\partial p_{2}}{\partial \alpha}}  - p_{2} \frac{\partial p_{7}}{\partial a} +  \bar{p}_{2} \frac{\partial p_{7}}{\partial \bar{a}} -
p_{7}\overline{\frac{\partial p_{2}}{\partial a}}  +\bar{p}_{7}\overline{\frac{\partial p_{2}}{\partial \bar{a}}}  \\
&& \qquad  \qquad \qquad  - p_{7}\left(\frac{\partial p_{7}}{\partial \bar{a}}  -
\overline{\frac{\partial p_{7}}{\partial \bar{a}} } \right)  +\frac{\partial p_{8}}{\partial a}  , \\
&& p_{10a} (\alpha,a,\bar{a})=  p_{8} \left(\frac{\partial p_{7}}{\partial \bar{a}} +\overline{\frac{\partial p_{7}}{\partial \bar{a}}}\right)+  \frac{\partial p_{8}}{\partial \alpha}, \\
&& p_{10b} (\alpha,a,\bar{a})= p_{1}(-p_{2}p_{7} + \overline{p_{2}p_{7} }) - p_{2} \frac{\partial p_{8}}{\partial a} + \bar{p}_{2}\frac{\partial p_{8}}{\partial \bar{a}}	 - p_{7}\frac{\partial p_{2}}{\partial \alpha} + \overline{p_{7}\frac{\partial p_{2}}{\partial \alpha}}  \\
&&  \qquad \qquad  \qquad \qquad +p_{8}\left(\frac{\partial p_{2}}{\partial a} -\overline{\frac{\partial p_{2}}{\partial a}} -\frac{\partial p_{7}}{\partial \bar{a}} +\overline{\frac{\partial p_{7}}{\partial \bar{a}}} \right )  , \\
&& p_{11} (\alpha,a,\bar{a})=	\frac{p_{9a}+p_{9b}}{2(|p_{7}|^2+p_{8})}, \\
&& p_{12}(\alpha,a,\bar{a})= \frac{1}{2(|p_{7}|^2+p_{8})}\left(p_{7}^2 \frac{\partial p_{2}}{\partial \bar{a}} + p_{7}(\bar{p}_{9a} -  \bar{p}_{9b}) -\bar{p}_{7} (p_{9a}+p_{9b})+ p_{10a} + p_{10b}\right), \\
&& p_{13}(\alpha,a,\bar{a})= \frac{1}{2(|p_{7}|^2+p_{8})}\left( p_{7}p_{8} \frac{\partial p_{2}}{\partial \bar{a}}  - \bar{p}_{7}(p_{10a}+p_{10b}) + p_{8}(\bar{p}_{9a} - \bar{p}_{9b})\right), 
\end{eqnarray*}
\begin{eqnarray*} 	
&& p_{14}(\alpha,a,\bar{a})= p_{1} - \frac{\partial p_{2}}{\partial a} + p_{12}, \\
&&  p_{15}(\alpha,a,\bar{a})= - p_{1}p_{2}  - \frac{\partial p_{2}}{\partial \alpha} + p_{13}, \\
&&	p_{16a}(\alpha,a,\bar{a})= p_{11}(p_{12}+2p_{14}) - p_{2} \frac{\partial p_{11}}{\partial a}  + \bar{p}_{2} \frac{\partial p_{11}}{\partial \bar{a}}  - p_{7} \frac{\partial p_{11}}{\partial \bar{a}}  +  \frac{\partial p_{12}}{\partial a}, \\	
&&	p_{16b}(\alpha,a,\bar{a})=   p_{11}( p_{1} +2p_{12} +p_{14} ) + \frac{\partial p_{11}}{\partial \alpha}  + p_{7} \frac{\partial p_{11}}{\partial \bar{a} } +  \frac{\partial p_{14}}{\partial a}, \\
&&	p_{16}(\alpha,a,\overline{a})=  	p_{16a} -  	p_{16b}, \\
&& p_{17a}(\alpha,a,\overline{a}) =  2p_{11}p_{15} + p_{12}p_{14} + \frac{1}{4}\left |\frac{\partial p_{2}}{\partial \overline{a}}\right |^2 + \frac{1}{2}  p_{7} \frac{\partial^2 p_{2}}{\partial \overline{a}\partial a}  -p_{7}p_{11} \frac{\partial p_{2}}{\partial \overline{a}}\\
&& \qquad \qquad  \qquad   -(|p_{7}|^2+p_{8}) \frac{\partial p_{11}}{\partial \overline{a}}  -p_{2} \frac{\partial p_{12}}{\partial a} + \overline{p_{2}} \frac{\partial p_{12}}{\partial \overline{a}}  - p_{7} \frac{\partial p_{12}}{\partial \overline{a}}  +  \frac{\partial p_{13}}{\partial a}  ,       \\
&&  p_{17b} (\alpha,a,\overline{a}) = p_{1}p_{14} + 2p_{11}p_{13} + p_{12}p_{14}  -  \frac{1}{4}\left |\frac{\partial p_{2}}{\partial \overline{a}}\right |^2  - \frac{1}{2}p_{7}  \frac{\partial^2 p_{2}}{\partial \overline{a}\partial a} + p_{7}p_{11}  \frac{\partial p_{2}}{\partial \overline{a}}   \\
&& \qquad \qquad  \qquad    + (|p_{7}|^2 + p_{8}) \frac{\partial p_{11}}{\partial \overline{a}} 
+ p_{7} \frac{\partial p_{14}}{\partial \overline{a}}  +  \frac{\partial p_{14}}{\partial \alpha} +   \frac{\partial p_{15}}{\partial a} ,\\ 
&& p_{17}(\alpha,a,\overline{a}) =  p_{17a} -  p_{17b},   \\
&&  p_{18a}(\alpha,a,\overline{a}) =  -\frac{1}{2}(2\overline{p_{7}} +2\overline{p_{7}}p_{11}  +p_{12} + \overline{p_{14}})\frac{\partial p_{2}}{\partial \overline{a} } + \frac{1}{2}(-p_{2} + \overline{p_{7}}) \frac{\partial^2  p_{2}}{\partial \overline{a} \partial a }   +  \frac{1}{2} \overline{p_{2}} \frac{\partial^2  p_{2}}{\partial \overline{a}^2 } \\
&&  \qquad \qquad  \qquad  \quad - \overline{p_{7}}^2 \frac{\partial p_{11}}{\partial \overline{a}}  - \frac{\partial p_{13}}{\partial\overline{a}},   \\
&&  p_{18b}(\alpha,a,\overline{a})  =  -\frac{1}{2}(p_{1}-2\overline{p_{7}} p_{11}   + \overline{p_{12}} - p_{14})\frac{\partial p_{2}}{\partial \overline{a}}  - \frac{1}{2} \frac{\partial^2 p_{2}}{\partial \alpha\partial a}  - \frac{1}{2} \overline{p_{7}}\frac{\partial^2 p_{2}}{\partial \overline{a}\partial a}  \\
&&  \qquad \qquad  \qquad  \quad  + \overline{p_{7}}^2\frac{\partial p_{11}}{\partial \overline{a}} + \overline{p_{7}} \frac{\partial p_{14}}{\partial \overline{a}}  +\frac{\partial p_{15}}{\partial \overline{a}} ,   \\
&& p_{18}(\alpha,a,\overline{a}) = p_{18a} -  p_{18b}, \\
&&  p_{19a}(\alpha,a,\overline{a})  = p_{12}p_{15}  - \frac{1}{2}( 2p_{8}p_{11} + \overline{p_{15}})\frac{\partial p_{2}}{\partial \overline{a}}  + \frac{1}{2}p_{8}\frac{\partial^2 p_{2}}{\partial \overline{a} \partial a}  - \overline{p_{7}p_{8}} \frac{\partial p_{11}}{\partial \overline{a}}  -p_{8}\frac{\partial p_{12}}{\partial \overline{a}}  \\
&& \qquad \qquad  \qquad - p_{2}  \frac{\partial p_{13}}{\partial a} + \overline{p_{2}}  \frac{\partial p_{13}}{\partial \overline{a}} ,\\
&&  p_{19b}(\alpha,a,\overline{a})  = p_{1}p_{15} + p_{13}p_{14} +\frac{1}{2} (2p_{8}p_{11} - \overline{p_{13}})\frac{\partial p_{2}}{\partial \overline{a}} -\frac{1}{2}p_{8} \frac{\partial^2 p_{2}}{\partial \overline{a}\partial a}  + \overline{p_{7}p_{8}} \frac{\partial p_{11}}{\partial \overline{a}}  + p_{8} \frac{\partial p_{14}}{\partial \overline{a}}  + \frac{\partial p_{15}}{\partial \alpha} ,  \\
&&  p_{19}(\alpha,a,\overline{a}) = p_{19a}- p_{19b},  \\
&&  p_{20}(\alpha,a,\overline{a}) = -\overline{p_{7}}p_{16} + p_{17}, \\
&&  p_{21}(\alpha,a,\overline{a})   = -\overline{p_{7}}p_{17} + p_{7}p_{18} + p_{19},  \\
&&  p_{22}(\alpha,a,\overline{a}) = p_{8}p_{18} - \overline{p_{7}}p_{19},  \\
&&  p_{23} (\alpha,a,\overline{a})=  \mbox{ a solution of} \ (3.22),  \\
&&   p_{24} (\alpha,a,\overline{a})=    \frac{\frac{\partial p_{16}}{\partial \alpha}p_{23}^3 + \frac{\partial p_{20}}{\partial \alpha} p_{23}^2   + \frac{\partial p_{21}}{\partial \alpha}p_{23} + \frac{\partial p_{22}}{\partial \alpha} }{3p_{16}p_{23}^2 + 2 p_{20}p_{23} + p_{21}} , \\
&&   p_{25} (\alpha,a,\overline{a}) =  \frac{\frac{\partial p_{16}}{\partial a}p_{23}^3 + \frac{\partial p_{20}}{\partial a} p_{23}^2   + \frac{\partial p_{21}}{\partial a}p_{23} + \frac{\partial p_{22}}{\partial a} }{3p_{16}p_{23
	}^2 + 2 p_{20}p_{23} + p_{21}} ,\\
&&   p_{26} (\alpha,a,\overline{a})=   \frac{\frac{\partial p_{16}}{\partial \overline{a}}p_{23}^3 + \frac{\partial p_{20}}{\partial \overline{a}} p_{23}^2   + \frac{\partial p_{21}}{\partial \overline{a}}p_{23} + \frac{\partial p_{22}}{\partial\overline{a}} }{3p_{16}p_{23}^2 + 2 p_{20}p_{23} + p_{21}} ,\\
&& \overline{p_{i}}(\alpha, a,\overline{a})= p_{i}(\alpha, \overline{a},a), \ (i=1,2,3,...).
	 \end{eqnarray*}
	 
	  \begin{lemma}
	 	$G(\alpha,a,\overline{a})$ is not identically zero.
	 \end{lemma} 
	 \begin{proof}
	 	In order to determine whether	$G(\pi/4,0,0) \neq 0$, we provide  a procedure for writing   the   Mathematica (by Wolfram) program codes.  First, we find the value  $p_{23}(\pi/4,0,0)$  by solving  (3.22) as follows.  Since all coefficients of (3.22) are real  at $(\pi/4,0,0)$ and $a_{\alpha} \neq \overline{a_{\alpha}}$, \  $p_{23}(\pi/4,0,0)$ is not zero and  uniquely determined up to the complex conjugate.   Next, by (3.10) and (3.22), we have
	 	$	p_{16}a_{\alpha}^3 + p_{20}a_{\alpha}^2 + p_{21}a_{\alpha} + p_{22}=0.$
	 	Partially differentiating   this  for $\alpha$, we obtain  $\partial p_{23}/\partial \alpha =-p_{24}$.
	 	We  note that $3p_{16}p_{23}^2 + 2 p_{20}p_{23} + p_{21}$   at   $(\pi/4,0,0)$  is not zero.  By the  similar way, we get  $\partial p_{23}/\partial a = -p_{25}$,  and $\partial p_{23}/\partial \overline{a} = -p_{26}$.   Consequently,  
	 	\begin{eqnarray*}
	 		&& G(\alpha,a,\overline{a}) = p_{1}p_{2} + \frac{\partial p_{2}}{\partial \alpha} - \left (p_{1} -\frac{\partial p_{2}}{\partial a} \right )p_{23}  +\frac{\partial p_{2}}{\partial \overline{a}} \overline{p_{23}}+ p_{24}  \\
	 		&&\qquad \qquad \qquad  + p_{2}p_{25} - (\overline{p_{2} }- 2\overline{p_{23}}) p_{26}.
	 	\end{eqnarray*}
	 	By this formula, we observe 
	 	$G(\pi/4,0,0) \neq 0$,  thereby proving Lemma 4.1.
	 \end{proof}

\medskip
\begin{flushleft}

 Katsuei  Kenmotsu \\
 Tohoku University,
  \,  Sendai, Japan \\
email:  kenmotsu-math@tohoku.ac.jp
\end{flushleft}
\end{document}